\documentclass[preprint,12pt]{elsarticle}

\usepackage{amsmath}
\usepackage{amsthm}

\usepackage{graphicx}
\usepackage{mathptmx}
\usepackage{latexsym}
\usepackage{amssymb}
\usepackage{dsfont}

\usepackage{xcolor}
\newcounter{example}[section]
\newenvironment{example}[1][]{\refstepcounter{example}\par\medskip
	\noindent \textbf{Example~\theexample. #1} \rmfamily}{\medskip}
\usepackage{mathptmx}      
\usepackage{mathtools}
\newcommand{\be}{\begin{equation}}
\newcommand{\ee}{\end{equation}}
\newcommand{\ben}{\begin{eqnarray*}}
	\newcommand{\een}{\end{eqnarray*}}
\newtheorem{examp}{\sc Example}
\newtheorem{remk}{\sc Remark}
\newtheorem{corol}{\sc Corollary}
\newtheorem{lemma}{\sc Lemma}
\newtheorem{theorem}{\sc Theorem}

\newtheorem{defn}{\sc Definition}
\newcommand{\bt}{\begin{theorem}}
	\newcommand{\et}{\end{theorem}}

\newtheorem{proposition}{Proposition}[section]
\newtheorem{definition}{Definition}[section]
\newtheorem{remark}{Remark}[section]

\newcommand{\bl}{\begin{lemma}}
	\newcommand{\el}{\end{lemma}}
\newcommand{\bed}{\begin{defn}}
	\newcommand{\eed}{\end{defn}}
\newcommand{\brem}{\begin{remk}}
	\newcommand{\erem}{\end{remk}}
\newcommand{\bex}{\begin{examp}}
	\newcommand{\eex}{\end{examp}}
\newcommand{\bcl}{\begin{corol}}
	\newcommand{\ecl}{\end{corol}}

\newcommand{\bea}{\begin{eqnarray}}
\newcommand{\eea}{\end{eqnarray}}

\journal{Journal of Optimization Theory and Applications}

\begin{document}

\begin{frontmatter}

\author{Bharat Pratap Chauhan\corref{cor}}
\cortext[cor]{Corresponding Author}\ead{bc102@snu.edu.in}
\author{Dipti Dubey}
\ead{dipti.dubey@snu.edu.in}
\address{Department of Mathematics\\ 
Shiv Nadar Institution of Eminence\\
Delhi-NCR 201314, India\\}

\title{On Almost-Type Special Structured Tensor Classes Associated with Semi-Positive Tensors}

\begin{abstract} In this paper, we introduce almost (strictly) semi-positive tensors, which extend the concept of almost (strictly) semimonotone matrices. Furthermore, we provide insights into the characteristics of the entries within these almost (strictly) semi-positive tensors and establish a condition that is both necessary and sufficient for categorizing the underlying tensor as an almost semi-positive tensor. Drawing inspiration from H. V\"{a}liaho's work on copositivity, we present the concept of almost (strictly) copositive tensors, which extends the notion of almost (strictly) copositive matrices to tensors. It is shown that a real symmetric tensor is  almost (strictly) semi-positive if and only if it is almost (strictly) copositive and a symmetric almost (strictly) semi-positive tensor has a (nonpositive) negative $H^{++}$-eigenvalue. We also establish a relationship between (strictly) diagonally dominant and (strong) $\mathcal{M}$-tensors with (strictly) semi-positive tensors.
\end{abstract}

\begin{keyword} Almost semi-positive tensors\sep Almost semimonotone matrices\sep Semi-positive tensors, Copositive tensors\sep Tensor complementarity problem.\\
AMS Classification: 15A69, 15B48, 90C33
\end{keyword}

\end{frontmatter}

\section{Introduction}

In multilinear algebra, several topics, including tensors, have attracted considerable attention in recent years. Tensors are generalizations to the concept of matrices in multilinear algebra. A multi-dimensional array or a hyper-matrix is referred to as a ``tensor'' in this context. A tensor can be seen as an extension of a matrix to higher orders, represented by:
\[
\mathcal{A} = (a_{i_1 \ldots i_m}), \quad a_{i_1 \ldots i_m} \in \mathbb{R}, \quad 1 \leq i_1, \ldots, i_m \leq n.
\]
The array $\mathcal{A}$, which is multidimensional, is an $m$th order tensor in $n$ dimensions, comprised of real numbers. It contains $n^m$ entries, denoted by $a_{i_1 \ldots i_m}$. As interest in multilinear algebra increases, especially in the domain of higher order tensors, the application area of structured matrices \cite{bpc,cps} has expanded to include higher order situations. Specially structured tensors play a significant role in tensor theory. Song and Qi \cite{sy} extended the concepts of (strictly) semimonotone matrices to (strictly) semi positive tensors. Semi positive tensors are those real tensors for which the operation $\mathcal{A}\mathbf{x}^{m-1}$ does not negate all positive entries of any nonzero, entrywise nonnegative vector $\mathbf{x}$, while a strictly semi positive tensor is a real tensor for which the operation $\mathcal{A}\mathbf{x}^{m-1}$ does not negate or turn to zero all positive entries of any nonzero, entrywise nonnegative vector $\mathbf{x}$ (see \cite{sy}). The class of (strictly) semi positive tensors is a generalization of the class of (strictly) copositive tensors, which contains (positive) nonnegative tensors. Also, each $\mathbf{P_0}$-tensor is certainly semi positive and each $\mathbf{P}$-tensor is strictly semi positive. For the basic properties and applications of the class of (strictly) copositive, and ($\mathbf{P}$ ) $\mathbf{P_0}$-tensors one can refer to \cite{copos,ptensor,Lqi}.

The concept of the tensor complementarity problem (TCP) was initially introduced by Song and Qi \cite{song,sy}. The tensor complementarity problem (TCP), a special case of the nonlinear complementarity problem, has garnered attention in the existing literature. Huang and Qi \cite{nonc} formulated an n-person noncooperative game as the tensor complementarity problem with a given nonnegative tensor. Consider a tensor $\mathcal{A} \in \mathbb{R}_{m, n}$ and a vector $\mathbf{q} \in \mathbb{R}^{n}$. The problem known as the tensor complementarity problem (TCP), denoted $\operatorname{TCP}(\mathbf{q}, \mathcal{A})$, seeks to either identify a vector $\mathbf{x} \in \mathbb{R}^{n}$ satisfying the conditions:
\begin{center}
$\mathbf{0} \leq \mathbf{x}, \quad \mathbf{0} \leq \mathbf{q}+\mathcal{A} \mathbf{x}^{m-1}, \quad \left\langle\mathbf{x}, \mathbf{q}+\mathcal{A} \mathbf{x}^{m-1}\right\rangle = \mathbf{0},$
\end{center}
or to establish the non-existence of such a vector. It is evident that $\operatorname{TCP}(\mathbf{q}, \mathcal{A})$ serves as a natural extension of the linear complementarity problem (when $m = 2$). Comprehensive insights into the properties and fundamental aspects of tensor complementarity problems can be found in references \cite{tcp1,tcp2,tcp3,sy}.

The class of (strictly) semi positive tensors, along with the properties of various structured tensors, plays a crucial role in the theory and study of the tensor complementarity problem \cite{Ysong}. A significant challenge with (strictly) semipositive tensors is the difficulty in both constructing and detecting them. In this paper, we construct special structured tensors whose proper principal subtensors maintain these structured tensor properties. Therefore, as we aim to establish a new characterization of (strictly) semipositive tensors, we find it beneficial to extend the concept of \textit{almost (strictly) semimonotone matrices} to \textit{almost (strictly) semipositive tensors}. An almost (strictly) semi positive tensor is one that is not (strictly) semi positive but has proper principal subtensors that are (strictly) semi positive. For further insight into the properties and basic facts of \textit{almost (strictly) semimonotone} matrices, refer to \cite{bpc,tsat,wen}. Additionally, for a comprehensive understanding of the basic properties and applications of the class of (strictly) semi positive tensors, see \cite{song,strict,semi}.

Inspired by the above studies of tensors, our main contributions in this paper are as follows.
\begin{itemize}
    \item First, we establish a relationship between (strictly) diagonally dominant and (strong) $\mathcal{M}$-tensors with (strictly) semi positive tensors. Moreover, a well-known result is that $A$ is a semimonotone matrix ($\mathbf{E_0}$) if and only if $A$ is a completely $\mathbf{S_{0}}$-matrix, (see \cite{cps}). The following question is natural. Can we extend this result of matrices to tensors? 
\item Next, we formally introduce a class of almost (strictly) semi positive tensors. We provide insights into the characteristics of the entries within these almost (strictly) semi positive tensors. Furthermore, we establish a necessary and sufficient condition for categorizing the underlying tensor as an almost semi positive tensor.
\item Lastly, drawing inspiration from H. V\"{a}liaho's work on copositivity \cite{H,val}, we present the concept of {\it almost (strictly) copositive tensors}, which extends the notion of {\it almost (strictly) copositive matrices} to tensors. These tensors play a pivotal role in establishing criteria for copositivity. It is established that for a real symmetric tensor, being almost (strictly) semi positive is equivalent to being almost (strictly) copositive. We also provide a non-trivial counterexample to show that the above result holds only under the symmetric assumption.
\end{itemize}

The rest of this paper is structured as follows. Section \ref{sec:preliminaries} presents fundamental symbols, definitions, and result. In Section \ref{sec:Semi positive tensors}, we explore several results associated with tensors that are (strictly) semi positive. Then, in Section \ref{sec:Almost (strictly) semi positive tensors}, we introduce the notion of almost (strictly) semi positive tensors and explore the properties related to their entries.
 Additionally, in Section \ref{sec: Main results}, we introduce the notion of almost (strictly) copositive tensors and establish their connection with almost (strictly) semi positive tensors.

\section{Preliminaries}\label{sec:preliminaries}

This section provides definitions and notations that will prove beneficial for our exploration. In this document, we denote scalars with small letters $x, y, \ldots$, vectors with small bold letters $\mathbf{x}, \mathbf{y}, \ldots$, matrices with capital letters $A, B, \ldots$, and tensors with calligraphic letters $\mathcal{A}, \mathcal{B}, \ldots$. It is important to note that all tensors, matrices, and vectors discussed in this document consist of real entries. The $n$-dimensional Euclidean space equipped with the usual inner product is denoted by $\mathbb{R}^n$. We write $\mathbf{x} \geq \mathbf{0}$ $(\mathbf{x} > \mathbf{0})$ if every component of vector $\mathbf{x}$ is nonnegative (positive). Let  $I_{n}:=\{1, \ldots, n\}$, and $\mathbb{R}^{n}:=\left\{\mathbf{x} = (x_1, \ldots, x_n)^{T} : x_i \in \mathbb{R}, i \in I_n\right\}$, $\mathbb{R}_{+}^{n}:=\left\{\mathbf{x} \in \mathbb{R}^{n}: \mathbf{x} \geq \mathbf{0}\right\}$, $\mathbb{R}^{n}_{++}:=\left\{\mathbf{x} \in \mathbb{R}^{n}: \mathbf{x} > \mathbf{0}\right\}$, where $\mathbb{R}$ is the set of real numbers \cite{tensor}. Any vector $\mathbf{x}\in \mathbb{R}^{n}$ is a column vector unless otherwise specified, and $\mathbf{x}^{T}$ denotes the transpose of $\mathbf{x}$.

For a square matrix $A \in \mathbb{R}^{n \times n}$ and a vector $\mathbf{q} \in \mathbb{R}^{n}$, the linear complementarity problem (LCP) is concerned with finding a vector $\mathbf{x} \in \mathbb{R}^{n}$ such that
\[
\mathbf{0} \leq \mathbf{x}, \quad \mathbf{0} \leq \mathbf{q} + A \mathbf{x}, \quad \langle\mathbf{x}, \mathbf{q} + A \mathbf{x} \rangle = \mathbf{0},
\]
or proving the non-existence of such a vector. This problem is denoted as $\operatorname{LCP}(\mathbf{q}, A)$ (see \cite{cps}).

\begin{definition}
 A real matrix $A \in \mathbb{R}^{n \times n}$ is said to be

    \begin{enumerate}
        \item semimonotone \cite{tsat}, denoted by $\mathbf{E_0}$, if for each $\mathbf{0 \neq x \geq 0}$, there exists an index $k \in I_n$ such that $x_k > 0$ and $(A\mathbf{x})_k \geq 0$;
        
        \item strictly semimonotone \cite{tsat}, denoted by $\mathbf{E}$, if for each $\mathbf{0 \neq x \geq 0}$, there exists an index $k \in I_n$ such that $x_k > 0$ and $(A\mathbf{x})_k > 0$;

        \item almost (strictly) semimonotone \cite{wen} if all proper principal submatrices of $A$ are (strictly) semimonotone and there exists a vector $\mathbf{x > 0}$ such that $A\mathbf{x < 0}$ (resp. $A\mathbf{x \leq 0}$);
        
        \item copositive \cite{cps} if $\mathbf{x}^{T}A\mathbf{x} \geq 0$ for all $\mathbf{x} \in \mathbb{R}_{+}^{n}$; and strictly copositive if $\mathbf{x}^{T}A\mathbf{x} > 0$ for all $\mathbf{x} \in \mathbb{R}_{+}^{n}\backslash \{\mathbf{0}\}$;
        
        \item almost (strictly) copositive \cite{val} if it is (strictly) copositive of order $(n-1)$ but not of order $n$;

        \item $\mathbf{Q}$-matrix \cite{cps} if $\operatorname{LCP}(\mathbf{q}, A)$ has a solution for all $\mathbf{q} \in \mathbb{R}^{n}$; and completely $\mathbf{Q}$-matrix, denoted by $\mathbf{\bar Q}$, if $A$ and all of its principal submatrices are $\mathbf{Q}$-matrices;
        
        \item $\mathbf{S}$-matrix \cite{cps} if there exists a vector $\mathbf{x > 0}$ such that $A\mathbf{x > 0}$; and completely $\mathbf{S}$-matrix, denoted by $\mathbf{\bar S}$, if $A$ and all of its principal submatrices are $\mathbf{S}$-matrices;
        
        \item $\mathbf{S_{0}}$-matrix \cite{cps} if there exists a vector $\mathbf{0 \neq x \geq 0}$ such that $A\mathbf{x \geq 0}$; and completely $\mathbf{S_0}$-matrix, denoted by $\mathbf{\bar S_{0}}$, if $A$ and all of its principal submatrices are $\mathbf{S_0}$-matrices.
    \end{enumerate}
\end{definition}

A real tensor $\mathcal{A}$ of order $m$ and dimension $n$, denoted as $\mathcal{A}=\left(a_{i_{1} \ldots i_{m}}\right)$, represents a multi-dimensional array with entries $a_{i_{1} \ldots i_{m}} \in \mathbb{R}$, where $i_{j} \in I_{n}$ for $j \in I_{m}$. If the entries $a_{i_1 \ldots i_m}$ of real tensor $\mathcal{A}$ in $\mathbb{R}_{m, n}$ remain unchanged under any permutation of their indices $\{i_1 \ldots i_m\}$, then $\mathcal{A}$ is termed as \textit{symmetric}. The collection of all the real tensors and real symmetric tensors of order $m$ and dimension $n$ is denoted as $\mathbb{R}_{m, n}$ and $S_{m, n}$, respectively. The identity tensor and zero tensor in $\mathbb{R}_{m, n}$ are represented by $\mathcal{I}$ and $\mathcal{O}$, respectively.

Consider a tensor $\mathcal{A}=\left(a_{i_{1} \ldots i_{m}}\right) \in \mathbb{R}_{m, n}$ and a vector $\mathbf{x} \in \mathbb{R}^{n}$. The operation $\mathcal{A} \mathbf{x}^{m-1}$ results in a vector in $\mathbb{R}^{n}$, where each element is calculated as
$$
\left(\mathcal{A} \mathbf{x}^{m-1}\right)_{i}:=\sum_{i_{2}, \ldots, i_{m}=1}^{n} a_{i i_{2} \ldots i_{m}} x_{i_{2}}\ldots x_{i_{m}}, \quad \forall i \in I_{n}.
$$
Moreover, $\mathcal{A} \mathbf{x}^{m}$ defines a homogeneous polynomial of degree $m$, given by
$$
\mathcal{A} \mathbf{x}^{m}:=\left\langle\mathbf{x}, \mathcal{A} \mathbf{x}^{m-1}\right\rangle=\sum_{i_{1}, i_{2}, \ldots, i_{m}=1}^{n} a_{i_{1} i_{2} \ldots i_{m}} x_{i_{1}} x_{i_{2}} \ldots x_{i_{m}}.
$$

A tensor $\mathcal{C} \in T_{m, r}$ is called a {\it principal sub-tensor} of a tensor $\mathcal{A}=\left(a_{i_1 \ldots i_m}\right) \in \mathbb{R}_{m, n}$ $(1 \leq r \leq n)$ if there is a set $J$ that composed of $r$ elements in $[n]$ such that
$$
\mathcal{C}=\left(a_{i_1 \ldots i_m}\right), \quad \text { for all } i_1, i_2, \ldots, i_m \in J
$$
We define $\mathcal{A}_r^J$ as the principal sub-tensor of a tensor $\mathcal{A} \in \mathbb{R}_{m, n}$, where the entries of $\mathcal{A}_r^J$ are indexed by $J \subseteq I_{n}$ with $|J|=r$ $(1 \leq r \leq n)$ \cite{lq} and, $\mathbf{x}_J$ denotes the $r$-dimensional sub-vector of a vector $\mathbf{x} \in \mathbb{R}^n$, with its components indexed by $J$. It is worth noting that diagonal entries are represented by the principal sub-tensors for $r = 1$.

Let $\mathcal{A}$ be a real tensor in $\mathbb{R}_{m, n}$. A complex valued element $\lambda$ is an {\it eigenvalue} of a real tensor $\mathcal{A}$ if there exists a non-zero vector $\mathbf{x}$ in $\mathbb{C}^n$ that solves the following system of equations:
\[
\mathcal{A}\mathbf{x}^{m-1} =\lambda \mathbf{x}^{[m-1]},
\] 
where $\mathbf{x}^{[m-1]} = (x_1^{m-1}, \ldots, x_n^{m-1})^T$. We refer to $\mathbf{x}$ as an \textit{eigenvector} of $\mathcal{A}$ corresponding to the eigenvalue $\lambda$. A particular eigenvalue is termed an \textit{H-eigenvalue} if it is real and possesses a real eigenvector $\mathbf{x}$, which is then termed an \textit{H-eigenvector}. An H-eigenvalue $\lambda$ of $\mathcal{A}$ is classified as:

(i) an H$^+$-eigenvalue of $\mathcal{A}$ if its associated H-eigenvector $\mathbf{x} \in \mathbb{R}^n_+$; and 

(ii) an H$^{++}$-eigenvalue of $\mathcal{A}$ if its associated H-eigenvector $\mathbf{x} \in \mathbb{R}^{n}_{++}$, (see \cite{nsc}).

\begin{definition}
Let $\mathcal{A}=\left(a_{i_{1} \ldots i_{m}}\right)$ denote a real tensor from $\mathbb{R}_{m, n}$. It is characterized as

\begin{enumerate}
    \item \textbf{semi positive} (denoted by $\mathbf{E_0}$-tensor) \cite{Ysong} if and only if for each $\mathbf{0 \neq x \geq 0}$, there exists an index $k \in I_{n}$ such that $x_{k}>0$ and $\left(\mathcal{A} \mathbf{x}^{m-1}\right)_{k} \geq 0$;
    \item \textbf{strictly semi positive} (denoted by $\mathbf{E}$-tensor) \cite{Ysong} if and only if for each $\mathbf{0 \neq x \geq 0}$, there exists an index $k \in I_{n}$ such that $x_{k}>0$ and $\left(\mathcal{A} \mathbf{x}^{m-1}\right)_{k}>0$;
    \item \textbf{copositive} (denoted by $\mathbf{C_0}$-tensor )\cite{Lqi} if and only if $\mathcal{A} \mathbf{x}^{m} \geq 0$ for all $\mathbf{x} \in \mathbb{R}_{+}^{n}$;
    \item \textbf{strictly copositive} (denoted by $\mathbf{C}$-tensor)\cite{Lqi} if and only if $\mathcal{A} \mathbf{x}^{m}>0$ for all $\mathbf{x} \in \mathbb{R}_{+}^{n} \backslash\{\mathbf{0}\}$;
     \item \textbf{$\mathbf{P_0}$-tensor}\cite{song} if and only if for any nonzero vector $\mathbf{x} \in \mathbb{R}^{n},$ there exists $i \in I_n$ such that $x_i \neq 0$ and
    \[
    x_i (\mathcal{A} \mathbf{x}^{m-1}>\mathbf{0})_i \geq 0;
    \]
    \item \textbf{$\mathbf{P}$-tensor}\cite{song} if and only if for any nonzero vector $\mathbf{x} \in \mathbb{R}^{n},$
    \[
    \max_{i \in I_n} x_i (\mathcal{A} \mathbf{x}^{m-1}>\mathbf{0})_i > 0;
    \]
    \item \textbf{$\mathbf{S_{0}}$-tensor}\cite{Gyu} if and only if the system
    \[
    \mathcal{A} \mathbf{x}^{m-1} \geq \mathbf{0}, \quad \mathbf{0 \neq x \geq 0}
    \]
    has a solution; and \textbf{completely $\mathbf{S_{0}}$-tensor} if and only if all its principal sub-tensors are $\mathbf{S_{0}}$-tensors;
    \item \textbf{S-tensor}\cite{mm,Gyu} if and only if the system
    \[
    \mathcal{A} \mathbf{x}^{m-1}>\mathbf{0}, \quad \mathbf{x}>\mathbf{0}
    \]
    has a solution; and \textbf{completely S-tensor} if and only if all its principal sub-tensors are $\mathbf{S}$-tensors.  
\end{enumerate}
\end{definition}

\begin{proposition}
Let $\mathcal{A}$ be a real $m$th order $n$-dimensional tensor. Then $\mathcal{A}$ is an S-tensor iff the system
\[
\mathcal{A} \mathbf{x}^{m-1} > \mathbf{0}, \quad \mathbf{x} \geq \mathbf{0}
\]
has a solution, (see Proposition 3.1. in \cite{Gyu}).
\end{proposition}


\section{Semi positive tensors}\label{sec:Semi positive tensors}

Song and Qi \cite{sy} extended the concepts of (strictly) semimonotone matrices to (strictly) semi positive tensors. In this section, we revisit the class of (strictly) semi positive tensors, denoted as ($\mathbf{E}$) $\mathbf{E_0}$-tensors and present some results associated with these tensors.

We initiate this section with certain propositions derived directly from the definition of (strict) semi positive tensors. The following two propositions play a fundamental role in characterizing the (strict) semi positivity of tensors. 
\begin{proposition}
A tensor $\mathcal{A} \in \mathbb{R}_{m, n}$ is considered $\mathbf{E_0}$-tensor if and only if the following conditions hold:
\begin{itemize}
\item [(a)] Each proper principal subtensor of $\mathcal{A}$ exhibits semi positivity.
\item [(b)] For any $\mathbf{x}$ where each component is greater than zero, $(\mathcal{A}\mathbf{x}^{m-1})_k \geq 0$ holds for at least one $k$.
\end{itemize}
\end{proposition}

In the case of $\mathbf{E}$-tensors, we have the following result:

\begin{proposition}
A tensor  $\mathcal{A} \in \mathbb{R}_{m, n}$ is strictly semi positive if and only if
\begin{itemize}
    \item [(a)] Every proper principal subtensor of $\mathcal{A}$ is $\mathbf{E}$-tensor, and
    \item [(b)] For any $\mathbf{x}$ where each component is greater than zero, $(\mathcal{A}\mathbf{x}^{m-1})_k > 0$ holds for at least one $k$.
\end{itemize}
\end{proposition}

The subsequent result shows that the set of (positive) nonnegative tensors constitutes a subset of ($\mathbf{E}$) $\mathbf{E_0}$-tensors.
\begin{proposition}
Let $\mathcal{A} \in \mathbb{R}_{m, n}$ be a tensor that is (positive) nonnegative. In that case, $\mathcal{A}$ is also a member of the set of tensors that are ($\mathbf{E}$) $\mathbf{E_0}$-tensors.
\end{proposition}
\begin{proof}
The proposition can be readily verified using the (strict) semi positivity of tensor.
\end{proof}

Let us discuss the {\it Hadamard product} (see \cite{hankel,sun}) of two semi positive tensors. Let $\mathcal{A} = (a_{i_1 \ldots i_m})$ and $\mathcal{B} = (b_{i_1 \ldots i_m})$ in $\mathbb{R}_{m, n}$. The Hadamard product of $\mathcal{A}$ and $\mathcal{B}$, denoted $\mathcal{A} \circ \mathcal{B}$, is defined as $\mathcal{A} \circ \mathcal{B} = (a_{i_1 \ldots i_m}b_{i_1 \ldots i_m})$ in $\mathbb{R}_{m, n}$.

We present the subsequent example to illustrate that the Hadamard product of two semi positive tensors may not necessarily yield a semi positive tensor.
\begin{example}
Let $\mathcal{A}=\left(a_{i_1 i_2 i_3}\right) \in \mathbb{R}_{3,2}$ be defined such that $a_{111}=1$, $a_{222}=1$,  $a_{211}=-1$,  $a_{122}=-1$, and all remaining elements $a_{i_1 i_2 i_3}=0$. Similarly, let $\mathcal{B}=\left(b_{i_1 i_2 i_3}\right) \in T_{3,2}$ be given by $b_{111}=1$, $b_{122}= 2$, $b_{211}=-2$, $b_{222}=1$, with all remaining $b_{i_1 i_2 i_3}=0$. Subsequently, for any real vector $\mathbf{x} = (x_1, x_2)^{T} \in \mathbb{R}^2$,
\[
(\mathcal{A} o \mathcal{B})\mathbf{x}^2 = \begin{pmatrix}
x_1^2 -  2x_2^2 \\
-2x_1^2 + x_2^2
\end{pmatrix}.
\]  
Verifying that both tensors $\mathcal{A}$ and $\mathcal{B}$ are semi positive is straightforward. However, the Hadamard product $\mathcal{A} o \mathcal{B}$ does not possess the property of semi-positivity. This can be demonstrated by considering $\mathbf{y} = (1, 1)^{T} > 0$, which leads to $(\mathcal{A} o \mathcal{B})\mathbf{y}^2 < 0$.
\end{example}

Now, we use the Hadamard product to present the relation between the class of ($\mathbf{E}$) $\mathbf{E_0}$-tensors and a subclass of the class of (positive) nonnegative tensors.
\begin{proposition}
If tensors $\mathcal{A}$ and $\mathcal{D}$ belong to $\mathbb{R}_{m, n}$ and $\mathcal{A}$ is (strictly) semi positive, while the diagonal tensor $\mathcal{D}$ possesses (positive) nonnegative diagonal elements, then the Hadamard product $\mathcal{A} \circ \mathcal{D}$ is also a ($\mathbf{E}$) $\mathbf{E_0}$-tensor.

\end{proposition}
\begin{proof}
With the definitions of the Hadamard product and the concept of ($\mathbf{E}$) $\mathbf{E_0}$-tensors, We can easily deduce the result.
\end{proof}

In the following result, we establish the relation between the $\mathbf{E_0}$-tensor and row-subtensor of the tensor. First, we recall the concept of a row-subtensor of a tensor.

\begin{definition}\cite{Jiayu}
 The $i$th row of a tensor $\mathcal{A} = (a_{i_1 \ldots i_m})$ of order $m$ and dimension $n$, denoted by $R_i(\mathcal{A})$, is the row-subtensor of order $m - 1$ and dimension $n$ of $\mathcal{A}$ with the entries
$$(R_i(\mathcal{A}))_{i_2 \ldots i_m} = a_{i i_2 \ldots i_m}.$$
\end{definition}

\begin{theorem}
Suppose that every proper principal subtensor of a tensor $\mathcal{A} \in \mathbb{R}_{m, n}$ is (strictly) semi positive, and if $\mathcal{A}$ contains a (positive) nonnegative row-subtensor, then tensor $\mathcal{A}$ is ($\mathbf{E}$) $\mathbf{E_0}$-tensor.
\end{theorem}
\begin{proof}
Let $R_{i}(\mathcal{A})$ denote the $i$-th row subtensor of tensor $\mathcal{A} \in \mathbb{R}_{m, n}$, characterized by nonnegative components. Consider any $\mathbf{0 \neq x \geq 0}$, where $\mathbf{x}$ may have either all positive entries or at least one zero entry. In the case of $\mathbf{x > 0}$, implying $x_i > 0$, we observe that
\[
(\mathcal{A}\mathbf{x}^{m-1})_i = \sum_{i_{2}, \ldots, i_{m}=1}^{n} (R_i(\mathcal{A}))_{i_2 \ldots i_m} x_{i_{2}} \ldots x_{i_{m}} = \sum_{i_{2}, \ldots, i_{m}=1}^{n} a_{i i_2 \ldots i_m} x_{i_{2}} \ldots x_{i_{m}} \geq 0
\]
$(\text{resp. } (\mathcal{A}\mathbf{x}^{m-1})_i > 0).$ 

Now, consider $\mathbf{0 \neq x \geq 0}$, where $\mathbf{x}$ has at least one entry equal to zero. Let us assume, without loss of generality, that for $j \in I_n$, ${x}_{j}=0$. This implies 
\[
\mathbf{0 \neq x} = (x_1, \ldots, x_{j-1}, 0, x_{j+1}, \ldots, x_n) \geq \mathbf{0}.
\]
Consequently, there is a $k \in I = I_n \backslash \{j\}$ with ${x}_k > 0$, where $|I| = r$ $(1 \leq r < n)$. It follows that 
\[
(\mathcal{A}\mathbf{x}^{m-1})_k = (\mathcal{A}_r^{I}(\mathbf{x}_I)^{m-1})_k \geq 0 \quad (\text{resp. } (\mathcal{A}\mathbf{x}^{m-1})_k > 0)
\]
since the proper principal subtensor $\mathcal{A}_r^I$ is ($\mathbf{E}$) $\mathbf{E_0}$-tensor. This proves the result.
\end{proof}

The next result shows the connection between semi positive tensors and diagonally dominant tensors. First, let us revisit the definition of (strictly) diagonally dominant tensors.

\begin{definition}
\cite{tensor} A tensor $\mathcal{A}$ is called diagonally dominant if
\[
\left|a_{i i \ldots i}\right| \geq \sum_{\substack{(i_2, \ldots, i_m) \\ \neq(i, \ldots, i)}} \left|a_{i i_2 \ldots i_m}\right| \quad \text{for all } i = 1, 2, \ldots, n,
\]

and called strictly diagonally dominant if
\[
\left|a_{i i \ldots i}\right| > \sum_{\substack{(i_2, \ldots, i_m) \\ \neq(i, \ldots, i)}} \left|a_{i i_2 \ldots i_m}\right| \quad \text{for all } i = 1, 2, \ldots, n.
\]
\end{definition}

\begin{theorem}\label{dd}
If all diagonal entries of a real tensor $\mathcal{A}$ in $\mathbb{R}_{m, n}$ are nonnegative and $\mathcal{A}$ is diagonally dominant, then $\mathcal{A}$ is $\mathbf{E_0}$-tensor.
\end{theorem}
\begin{proof}
Suppose that all diagonal entries of a real tensor $\mathcal{A}$ in $\mathbb{R}_{m, n}$ are nonnegative, and $\mathcal{A}$ is diagonally dominant. Then, for any nonzero and nonnegative vector $\mathbf{x} \in \mathbb{R}^n$, we select ${x}_k \geq {x}_j$ for every $j \in I_n$, where ${x}_k > 0$. By the definition of diagonal dominance, we have
\begin{align*}
(\mathcal{A} \mathbf{x}^{m-1})_k & = a_{k \ldots k} x_k^{m-1} + \sum_{\substack{i_2, \ldots, i_m=1, \\ \left(i_2, \ldots, i_m\right) \neq(k, \ldots, k)}}^n a_{k i_2 \ldots i_m} x_{i_2} \ldots x_{i_m} \\
& \geq \left(a_{k \ldots k} - \sum_{\substack{i_2, \ldots, i_m=1, \\ \left(i_2, \ldots, i_m\right) \neq(k, \ldots, k)}}^n\left|a_{k i_2 \ldots i_m}\right|\right)x_k^{m-1} \\
& \geq  0.
\end{align*}
Therefore, $\mathcal{A}$ is a $\mathbf{E_0}$-tensor.
\end{proof}

The theorem presented here can be established in a manner similar to Theorem \ref{dd}; therefore, we skip the proof.

\begin{theorem}
If the real tensor $\mathcal{A}$ in $\mathbb{R}_{m, n}$ has all its diagonal elements positive and is strictly diagonally dominant, then tensor $\mathcal{A}$ is $\mathbf{E}$-tensor.
\end{theorem}
Next, we shift our focus to the class of $\mathbf{Z}$-tensors, which are tensors with nonpositive off-diagonal elements and play a pivotal role in tensor theory. In the next theorem, we establish a relation between the (strong) $\mathcal{M}$-tensors (see \cite{mtensor,tensor,Mtensor}) and ($\mathbf{E}$) $\mathbf{E_0}$-tensors under the assumption of $\mathbf{Z}$-tensors. First, we recall the following definition:
\begin{definition}
\cite{mtensor} A tensor $\mathcal{A} \in \mathbb{R}_{m, n}$ is a $\mathbf{Z}$-tensor if all of its off-diagonal entries are nonpositive. Thus, a $\mathbf{Z}$-tensor $\mathcal{A}$ has a representation of the form $\mathcal{A} = t\mathcal{I - B}$, where $\mathcal{B}$ is a nonnegative tensor. If, in addition, $t \geq \rho(\mathcal{B})$, then we refer to $\mathcal{A}$ as an $\mathcal{M}$-tensor. Furthermore, an $\mathcal{M}$-tensor as described above is Nonsingular if and only if $t > \rho(\mathcal{B})$ and a nonsingular $\mathcal{M}$-tensor is also referred to as a strong $\mathcal{M}$-tensor. 
\end{definition}

\begin{theorem}
Let $\mathcal{A} \in \mathbb{R}_{m, n}$ represent a $Z$-tensor. Then, tensor $\mathcal{A}$ is ($\mathbf{E}$) $\mathbf{E_0}$-tensor iff it is (strong) $\mathcal{M}$.
\end{theorem}
\begin{proof}
Assume $\mathcal{A}$ to be a $Z$-tensor, failing to qualify as a (strong) $\mathcal{M}$-tensor. Consequently, $\mathcal{A}$ can be formulated as $\mathcal{A} = t\mathcal{I - B}$, where $\mathcal{B}$ is non-negative and $t < \rho(\mathcal{B})$ (resp. $t \leq \rho(\mathcal{B})$). Now, by the Perron-Frobenius Theorem (see \cite{land,yang}), $\rho(\mathcal{B})$ is an eigenvalue of $\mathcal{B}$ with an associated eigenvector $\mathbf{x}$ with $\mathbf{0 \neq x \geq 0}$. Thus, note that for $x_i >0$ where $i \in I_n$, we have
\[
(\mathcal{A} \mathbf{x}^{m-1})_i = (t \mathbf{x}^{m-1} - \mathcal{B} \mathbf{x}^{m-1})_i = (t - \rho(\mathcal{B})) x_i^{m-1} < 0,
\]
$(\text{resp. } (\mathcal{A} \mathbf{x}^{m-1})_i = (t - \rho(\mathcal{B})) x_i^{m-1} \leq 0)$. This implies that the tensor $\mathcal{A}$ is not  ($\mathbf{E}$) $\mathbf{E_0}$-tensor.

Conversely, let us assume $\mathcal{A}$ to be a (strong) $\mathcal{M}$-tensor. From \cite[Theorems 5.26 and 5.32]{tensor}, it follows that tensor $\mathcal{A}$ is ($P$) $P_0$, and hence, tensor $\mathcal{A}$ is ($\mathbf{E}$) $\mathbf{E_0}$-tensor.
\end{proof}

Now, we proceed to discuss a significant result of matrices into tensors in the following remarks.
\begin{remark}
Note that  the class of semimonotone matrices ( $\mathbf{E_0}$ ) and the class of completely $\mathbf{S_{0}}$-matrices are identical, that is, $\mathbf{E_0} =  {\mathbf{\bar{S}_0}}$ \cite{cps}.
This characterization of semimonotone matrices holds significant value in various aspects. However, the above result does not apply to tensors. In this instance, we construct a tensor that qualifies as a $\mathbf{\bar{S}_0}$-tensor but does not meet the criteria for being a semi positive tensor.

Consider $\mathcal{A}=\left(a_{i_1 i_2 i_3}\right) \in \mathbb{R}_{3,2}$, with specified entries $a_{111}=1$, $a_{222}=0$,  $a_{212}=-1$, $a_{112}=-1$, and all remaining entries $a_{i_1 i_2 i_3}=0$. Subsequently, for every real vector $\mathbf{x} = (x_1, x_2)^{T} \in \mathbb{R}^2$

\[
\mathcal{A} \mathbf{x}^2=\begin{pmatrix}
x_1^2-x_1 x_2 \\
-x_1 x_2
\end{pmatrix}.
\]
Here, each proper principal subtensor of tensor $\mathcal{A}$ is $\mathbf{S_{0}}$-tensor. Let $\mathbf{y} = (1, 0)^{T} \neq \mathbf{0}$ be a nonnegative vector. Then, we have $\mathcal{A} \mathbf{y}^2 \geq \mathbf{0}$, implying that tensor $\mathcal{A}$ is a $\mathbf{\bar S_{0}}$. However, for $\mathbf{x} = (1,2)^{T} > 0$, the following holds $\mathcal{A} \mathbf{x}^2 < 0$, and hence the tensor $\mathcal{A}$ is not semi positive.
\end{remark}

\begin{remark}
Note that the class of strictly semimonotone matrices  ($\mathbf E$) and the class of completely $\mathbf{S}$-matrices  are same, i.e.,  $\mathbf E =  \mathbf{ \bar{S}}$ \cite{cps}. In the example below, we find a $\mathbf{\bar{S}}$-tensor that is not semi positive. Consider $\mathcal{A}=\left(a_{i_1 i_2 i_3}\right) \in \mathbb{R}_{3,2}$, where the elements are defined such that $a_{111}=1$, $a_{112}=-2$, $a_{122}=1$, $a_{211}=-1$, $a_{222}=1$, and all the remaining entries are zero. Now, for any vector $\mathbf{x} = (x_1, x_2)^{T} \in \mathbb{R}^2$,
\[
\mathcal{A} \mathbf{x}^2=\begin{pmatrix}
x_1^2-2x_1 x_2+x_2^2 \\
-x_1^2 + x_2^2
\end{pmatrix}.
\]
Here, each proper principal subtensor of tensor $\mathcal{A}$ is $\mathbf{S}$-tensor. Take $\mathbf{y} = (1, 1.1)^{T}$, then $\mathbf{y > 0}$ and $\mathcal{A} \mathbf{y}^2 > \mathbf{0}$, thus, $\mathcal{A}$ is a $\mathbf{\bar S}$-tensor. However, for $\mathbf{y} = (1,1)^{T} > \mathbf{0}$, the following holds  $\mathcal{A} \mathbf{y}^2 = \mathbf{0}$, and hence the tensor $\mathcal{A}$ is not strictly semi positive.
\end{remark}


\section{Almost (strictly) semi positive tensors}\label{sec:Almost (strictly) semi positive tensors}
In this section, we first introduce the concept of the {\it almost (strictly) semi positive tensor}, which is a natural extension of an {\it almost (strictly) semimonotone matrix} (see \cite{bpc,wen}). Such tensors are of crucial importance in deriving criteria for
semi positively of tensors. Subsequently, we present a variety of results related to these tensor classes. Let us begin by stating the definition of a tensor with a special structure.
\begin{definition}\label{d1}
A real tensor $\mathcal{A} \in \mathbb{R}_{m, n}$ is {\it almost semi positive} if every proper principal subtensor of $\mathcal{A}$ is a semi positive tensor and there exists a real vector $\mathbf{x > 0}$ such that $(\mathcal{A}\mathbf{x}^{m-1}) < \mathbf{0}$. We denote such tensor by almost $\mathbf{E_0}$-tensor.
\end{definition}

In the example below, we construct an {\it almost $\mathbf{E_0}$-tensor}.

\begin{example}
Consider a tensor $\mathcal{A}$ defined as $\left(a_{i_1 i_2 i_3}\right) \in \mathbb{R}_{3,2}$, where the elements are specified as follows: $a_{111}=1$,  $a_{212}=-1$, $a_{112}=-1$, $a_{222}=0$, and all the remaining entries are zero. For any vector $\mathbf{x} = (x_1, x_2)^{T} \in \mathbb{R}^2$,
$$
\mathcal{A} \mathbf{x}^2=\left(\begin{array}{l}
x_1^2-x_1 x_2 \\
-x_1 x_2
\end{array}\right).
$$
We observe that all proper principal subtensors of $\mathcal{A}$ exhibit the semi positive property. Additionally, for $\mathbf{x} = (1, 2)^{T} > \mathbf{0}$, the following holds $\mathcal{A} \mathbf{x}^2 < \mathbf{0}$, and hence the tensor $\mathcal{A}$ is almost $\mathbf{E_0}$-tensor.
\end{example}

Next, we state the definition of another tensor with a special structure.
\begin{definition}\label{d2}
A real tensor $\mathcal{A} \in \mathbb{R}_{m, n}$ is said to be {\it almost strictly semi positive} if every proper principal subtensor of tensor $\mathcal{A}$ is strictly semi positive and there exists a real vector $\mathbf{x > 0}$ such that $(\mathcal{A}\mathbf{x}^{m-1}) \leq \mathbf{0}$. We denote such tensor by almost $\mathbf{E}$-tensor.
\end{definition}

In the example below, we construct an {\it almost $\mathbf{E}$-tensor}.

\begin{example}\label{eg1}
Consider a tensor $\mathcal{A}$ represented by $\left(a_{i_1 i_2 i_3}\right) \in \mathbb{R}_{3,2}$, where specific entries are defined such that $a_{111}=1$, $a_{112}=-1$, $a_{212}=-3$, $a_{222}=1$, and all the remaining entries are zero. For any vector $\mathbf{x} = (x_1, x_2)^{T} \in \mathbb{R}^2$,
\[
\mathcal{A} \mathbf{x}^2=\begin{pmatrix}
x_1\left(x_1-x_2\right) \\
x_2\left(x_2-3x_1\right)
\end{pmatrix}.
\]
It is apparent that all proper principal subtensors of $\mathcal{A}$ exhibit strict semi-positivity. Moreover, for the chosen vector $\mathbf{x} = (1,1)^{T} > \mathbf{0},$ the inequality $\mathcal{A} \mathbf{x}^2 \leq \mathbf{0}$ is satisfied. Thus, $\mathcal{A}$ qualifies as an almost $\mathbf{E}$-tensor.
\end{example}

The following remark is useful in studying the {\it almost $\mathbf{E}$-tensors}.

\begin{remark}\label{classification}
The almost $\mathbf{E}$-tensors $\mathcal{A}$ can be classified into two types. It is worth noting that if tensor $\mathcal{A}$ is almost $\mathbf{E}$, then tensor $\mathcal{A}$ is either almost $\mathbf{E_0}$ or $\mathcal{A}$ is both an almost $\mathbf{E}$-tensor and an $\mathbf{E_0}$-tensor. In Example \ref{eg1}, an almost $\mathbf{E}$-tensor $\mathcal{A}$ is almost $\mathbf{E_0}$, as evidenced by the fact that for $\mathbf{x} = (1,2)^{T} > \mathbf{0},$ we observe $\mathcal{A} \mathbf{x}^2 < \mathbf{0}$. Subsequently, we proceed to construct a tensor that possesses the characteristics of both an almost $\mathbf{E}$-tensor and a $\mathbf{E_0}$-tensor.

Consider a tensor $\mathcal{A}=\left(a_{i_1 i_2 i_3}\right) \in \mathbb{R}_{3,2}$, where specific elements are given such that $a_{111}=1$, $a_{222}=1$, $a_{122}=-1$, $a_{211}=-1$, and all the remaining elements are zero. For any real-valued vector $\mathbf{x} = (x_1, x_2)^{T} \in \mathbb{R}^2$, we have following expression:
\[
\mathcal{A} \mathbf{x}^2 = \begin{pmatrix}
x_1^2 - x_2^2 \\
x_2^2 - x_1^2
\end{pmatrix}.
\]
It is evident that every proper principal subtensor of tensor $\mathcal{A}$ is  an $\mathbf{E}$-tensor. Moreover, when considering $\mathbf{x}=(1,1)^{T}$, it is observed that $\mathcal{A} \mathbf{x}^2 = \mathbf{0}$. Therefore, $\mathcal{A}$ can be classified as an almost $\mathbf{E}$-tensor.

Moreover, considering any vector $\mathbf{x} = (x_1, x_2)^{T} \in \mathbb{R}^2$ with $\mathbf{x}$ being non-zero and each component nonnegative, the latter conditions hold: $x_1 < x_2$ or $x_2 \leq x_1$. Whenever $x_1 < x_2$, we have $x_2 > 0$ and $\left(\mathcal{A} \mathbf{x}^2\right)_2=x_2^2-x_1^2>0$, indicating that tensor $\mathcal{A}$ is a $\mathbf{E_0}$-tensor in this situation. On the other hand, if $x_2 \leq x_1$, then we have $x_1>0$ and $\left(\mathcal{A} \mathbf{x}^2\right)_1=x_1^2-x_2^2 \geq 0$, which implies that tensor $\mathcal{A}$ is a $\mathbf{E_0}$-tensor in this scenario and hence tensor $\mathcal{A}$ is a $\mathbf{E_0}$-tensor. 
 \end{remark}
\begin{remark} 
As expected, the sum and Hadamard product (see \cite{hankel}) of two almost $\mathbf{E_0}$-tensors does not necessarily result in an almost $\mathbf{E_0}$-tensor. To justify this statement, let us consider a nontrivial example.

Suppose $\mathcal{A}$ be represented by $\left(a_{i_1 i_2 i_3}\right) \in \mathbb{R}_{3,2}$, where specific elements are $a_{111}=0$, $a_{122}=a_{211}=-1/2$, $a_{222}=1$, and all the remaining entries are $a_{i_1 i_2 i_3}=0$. Additionally, consider $\mathcal{B}=\left(b_{i_1 i_2 i_3}\right) \in \mathbb{R}_{3,2}$ where $b_{111}=1, b_{122}=-1/2, b_{211}=-1/2, b_{222}=0$, and all the remaining entries $b_{i_1 i_2 i_3}=0$. For any vector $\mathbf{x} = (x_1, x_2)^{T} \in \mathbb{R}^2$, we have the following expression:
\[
\mathcal{(A + B)} \mathbf{x}^2 = \begin{pmatrix}
x_1^2 -  x_2^2 \\
-x_1^2 + x_2^2
\end{pmatrix} \quad \text{and} \quad (\mathcal{A}o\mathcal{B}) \mathbf{x}^2 = \frac{1}{4} \begin{pmatrix}
x_2^2 \\
x_1^2
\end{pmatrix}.
\]
Both tensors $\mathcal{A}$ and $\mathcal{B}$ are almost $\mathbf{E_0}$-tensors, but neither $\mathcal{A + B}$ nor $\mathcal{A}o\mathcal{B}$ is an almost $\mathbf{E_0}$-tensor.
\end{remark}
In the sequel, we present some results related to properties of entries of almost ($\mathbf{E}$) $\mathbf{E_0}$-tensors.

\begin{proposition}
Suppose that the tensor $\mathcal{A} = (a_{i_1 i_2 \ldots i_m}) \in \mathbb{R}_{m, n}$.
\begin{itemize}
    \item [(a)] If $\mathcal{A}$ is an almost $\mathbf{E_0}$-tensor, then each entry $a_{i i \ldots i}$ is greater than or equal to zero for all $i \in I_n$.
    \item [(b)] If $\mathcal{A}$ is an almost $\mathbf{E}$-tensor, then each entry $a_{i i \ldots i}$ is greater than zero for all $i \in I_n$.
    \item [(c)] If $\mathcal{A}$ is an almost $\mathbf{E_0}$-tensor, then there is a $k \in I_n$ satisfying \[ a_{k k \ldots k} + \sum_{\substack{0 > a_{k i_2 \ldots i_m} \\ (i_2, \ldots, i_m) \neq (k, \ldots, k)}} a_{k i_2 \ldots i_m} < 0. \]
    \item [(d)] If $\mathcal{A}$ is an almost $\mathbf{E}$-tensor, then there is a $k \in I_n$ satisfying \[ a_{k k \ldots k} + \sum_{\substack{0 > a_{k i_2 \ldots i_m} \\ (i_2, \ldots, i_m) \neq (k, \ldots, k)}} a_{k i_2 \ldots i_m} \leq 0. \]
\end{itemize}
\end{proposition}

\begin{proof}
The proofs for both the results $(a)$ and $(b)$ follow directly from Definitions \ref{d1} and \ref{d2}, by considering the vector $\mathbf{x}^{(i)}$ with a value of $1$ at the $i$th position for $i \in I_n$.

Next, we establish $(c)$. Assume that tensor $\mathcal{A}$ is an almost $\mathbf{E_0}$-tensor. This implies the existence of a positive real-valued vector $\mathbf{x}$ such that $(\mathcal{A} \mathbf{x}^{m-1})_i < 0$ holds for all $i \in I_n$. Select ${x}_j \leq {x}_k$ for each $j \in I_n$. Then, since ${x}_k > 0$, we have
\[
\begin{aligned}
& 0 > \left(\mathcal{A} \mathbf{x}^{m-1}\right)_k = \sum_{i_2, \ldots, i_m=1}^n a_{k i_2 \ldots i_m} {x}_{i_2} \ldots {x}_{i_m} \\
& = a_{k \ldots k} {x}_k^{m-1} + \sum_{\substack{0 < a_{k i_2 \ldots i_m} \\ \left(i_2, \ldots, i_m\right) \neq (k, \ldots, k)}} a_{k i_2 \ldots i_m} {x}_{i_2} \ldots {x}_{i_m} + \sum_{\substack{0 > a_{k i_2 \ldots i_m} \\ \left(i_2, \ldots, i_m\right) \neq (k, \ldots, k)}} a_{k i_2 \ldots i_m} {x}_{i_2} \ldots {x}_{i_m} \\
& \geq a_{k \ldots k} {x}_k^{m-1} + \sum_{\substack{0 > a_{k i_2 \ldots i_m} \\ \left(i_2, \ldots, i_m\right) \neq (k, \ldots, k)}} a_{k i_2 \ldots i_m} {x}_{i_2} \ldots {x}_{i_m} \\
& \geq a_{k \ldots k} {x}_k^{m-1} + \sum_{\substack{0 > a_{k i_2 \ldots i_m}  \\ \left(i_2, \ldots, i_m\right) \neq (k, \ldots, k)}} a_{k i_2 \ldots i_m} {x}_k^{m-1} \\
& = \left(a_{k \ldots k} + \sum_{\substack{0 > a_{k i_2 \ldots i_m} \\ \left(i_2, \ldots, i_m\right) \neq (k, \ldots, k)}} a_{k i_2 \ldots i_m}\right) {x}_k^{m-1}.
\end{aligned}
\]
This validates the assertion of part $(c)$.

To establish $(d)$, we can prove the claim for the almost strictly semi positive tensor by using a similar line of reasoning as in part $(c)$, making necessary adjustments to the inequalities.
\end{proof}
\begin{theorem}\label{al1}
Every row-subtensor of an almost $\mathbf{E_0}$-tensor contains at least one negative entry.
\end{theorem}
\begin{proof}
Assume that the tensor $\mathcal{A} = \left(a_{i_{1} \ldots i_{m}}\right) \in \mathbb{R}_{m, n}$ is almost $\mathbf{E_0}$-tensor that possesses a nonnegative row-subtensor $R_{i}(\mathcal{A}) = \left(a_{i i_{2} \ldots i_{m}}\right) \in \mathbb{R}_{m-1, n}$ for some $i \in I_{n}$. Subsequently, for any vector $\mathbf{x} \in \mathbb{R}^n$ that is nonzero and nonnegative, we have:
\[
\left(\mathcal{A} \mathbf{x}^{m-1}\right)_{i} = \sum_{i_{2}, \ldots, i_{m}=1}^{n} a_{i i_{2} \ldots i_{m}} x_{i_{2}} \ldots x_{i_{m}} \geq 0.
\]
However, this contradicts the existence of a positive real-valued vector $\mathbf{x}$ satisfying $(\mathcal{A}\mathbf{x}^{m-1}) < \mathbf{0}$. Hence, every row-subtensor of $\mathcal{A}$ contains a negative entry.
\end{proof}

The following result is true for almost $\mathbf{E}$-tensors.
	
\begin{theorem}
Every row-subtensor of an almost $\mathbf{E}$-tensor contains at least one negative entry.
\end{theorem}
\begin{proof}
Assume $\mathcal{A}$ is an almost $\mathbf{E}$-tensor. Then, from Remark \ref{classification}, $\mathcal{A}$ falls into one of two categories: either it is an almost $\mathbf{E_0}$-tensor, or it is both an almost $\mathbf{E}$-tensor and an $\mathbf{E_0}$-tensor. 

If $\mathcal{A}$ is an almost $\mathbf{E_0}$-tensor, then the conclusion straightforwardly follows from Theorem \ref{al1}. However, if $\mathcal{A}$ is both an almost $\mathbf{E}$-tensor and an $\mathbf{E_0}$-tensor, and there exists a positive real-valued vector $\mathbf{x}$ such that $\mathcal{A}\mathbf{x}^{m-1} = \mathbf{0}$, it implies that $\mathcal{A}$ cannot possess a row-subtensor with all entries being zero. This follows from the fact that the diagonal elements of an almost $\mathbf{E}$-tensor are positive, and thus, each row-subtensor of $\mathcal{A}$ must contain at least one negative entry.
\end{proof}

Next, we establish that almost ($\mathbf{E}$) $\mathbf{E_0}$-tensors maintain their original characteristics when multiplied by a diagonal matrix with positive entries on the diagonal.
\begin{theorem}
For any positive diagonal matrix $D \in \mathbb{R}^{n \times n}$ and tensor $\mathcal{A} \in \mathbb{R}_{m, n}$, the following three statements are considered equivalent:
\begin{itemize}
\item [(i)] tensor $\mathcal{A}$ is an almost ($\mathbf{E}$) $\mathbf{E_0}$-tensor.
\item[(ii)] tensor $D\mathcal{A}$ is an almost ($\mathbf{E}$) $\mathbf{E_0}$-tensor.
\item[(iii)] tensor $\mathcal{A}D$ is an almost ($\mathbf{E}$) $\mathbf{E_0}$-tensor.
\end{itemize}
\end{theorem}
\begin{proof}
To prove $(i) \Rightarrow (ii)$, let us assume $\mathcal{A}$ is an almost ($\mathbf{E}$) $\mathbf{E_0}$-tensor. Consequently, every proper principal subtensor of tensor $\mathcal{A}$ is a ($\mathbf{E}$) $\mathbf{E_0}$-tensor. Then, from Theorem 4 in \cite{Das}, every proper principal subtensor of $D \mathcal{A}$ is a ($\mathbf{E}$) $\mathbf{E_0}$-tensor. Now, since there exists a real-valued positive vector $\mathbf{x}$ in  $\mathbb{R}^{n}$ satisfying $\mathcal{A}\mathbf{x}^{m-1} < \mathbf{0}$ (resp. $\mathcal{A}\mathbf{x}^{m-1} \leq \mathbf{0}$), it implies that $D \mathcal{A}\mathbf{x}^{m-1} < \mathbf{0}$ (resp. $D \mathcal{A}\mathbf{x}^{m-1} \leq \mathbf{0}$). Hence, $D \mathcal{A}$ is an almost ($\mathbf{E}$) $\mathbf{E_0}$-tensor.

To show $(ii) \Rightarrow (i)$, we observe that $\mathcal{A} = D^{-1} D \mathcal{A}$ for any $D = \text{diag}(d_1, \ldots, d_n) \in \mathbb{R}^{n \times n}$, a diagonal matrix with the positive elements on the diagonal. Therefore, the assertion follows from the implication $(i) \Rightarrow (ii)$.

For $(ii) \iff (iii)$, let $\mathcal{A} = (a_{i_1 \ldots i_n}) \in \mathbb{R}_{m, n}$ and any $D = \text{diag}(d_1, \ldots, d_n) \in \mathbb{R}^{n \times n}$, a diagonal matrix with the positive elements $d_i$ on the diagonal, where $i \in I_n$. Then, by the definition of multiplication (see \cite{jy}), for any $i \in I_n$ and for $x \in \mathbb{R}^{n}$, we obtain,
\begin{align}
    (D\mathcal{A} \mathbf{x}^{m-1})_i &= \sum_{i_{2}, \ldots, i_{m}=1}^{n} d_{i} a_{i i_{2} \ldots i_{m}} x_{i_{2}} \ldots x_{i_{m}}, \label{eq:first_eq} \\
\text{and} \quad \left(\mathcal{A} D\mathbf{x}^{m-1}\right)_{i} &= 
 \sum_{i_{2}, \ldots, i_{m}=1}^{n} a_{i i_{2} \ldots i_{m}} d_{i_{2}} d_{i_{3}} \ldots d_{i_{m}} x_{i_{2}} x_{i_{3}} \ldots x_{i_{m}} \nonumber \\
&= \frac{d_i}{d_i} \sum_{i_{2}, \ldots, i_{m}=1}^{n} a_{i i_{2} \ldots i_{m}} \left(d_{i_{2}} x_{i_{2}}\right) \left(d_{i_{3}} x_{i_{3}}\right) \ldots \left(d_{i_{m}} x_{i_{m}}\right) \nonumber \\
&= \frac{1}{d_i} \left(\sum_{i_{2}, \ldots, i_{m}=1}^{n} d_i a_{i i_{2} \ldots i_{m}} y_{i_{2}} y_{i_{3}} \ldots y_{i_{m}}\right) \nonumber \\
&= \frac{1}{d_{i}} \left(D\mathcal{A} y^{m-1}\right)_{i} \label{eq:second_eq}
\end{align}
where $\mathbf{y} = D \mathbf{x} \in \mathbb{R}^n$. The desired result can now be obtained from Equations \ref{eq:first_eq} and \ref{eq:second_eq}.
\end{proof}

Next, our attention shifts towards showing the invariance characteristic of almost ($\mathbf{E}$) $\mathbf{E_0}$-tensors when the subscripts are rearranged.

\begin{theorem}
If any permutation matrix $P \in \mathbb{R}^{n \times n}$ and real tensor $\mathcal{A}$ in $\mathbb{R}_{m, n}$, then tensor $\mathcal{A}$ is almost ($\mathbf{E}$) $\mathbf{E_0}$-tensor iff $P\mathcal{A}P^{T}$ is almost ($\mathbf{E}$) $\mathbf{E_0}$-tensor.
\end{theorem}
\begin{proof}
Assume that matrix $P \in \mathbb{R}^{n \times n}$ is any permutation matrix that satisfies $(P\mathbf{z})_i = z_{\sigma(i)}$ for every real vector $\mathbf{z} \in \mathbb{R}^{n}$, where $i \in I_n$ and $\sigma$ is the permutation from $I_n$. Let the tensor $\mathcal{A} \in \mathbb{R}_{m, n}$ be an almost ($\mathbf{E}$) $\mathbf{E_0}$-tensor. Every proper principal subtensor of tensor $\mathcal{A}$ is thus a ($\mathbf{E}$) $\mathbf{E_0}$-tensor, and there is a positive $\mathbf{x} \in \mathbb{R}^{n}$ we have $\mathcal{A}\mathbf{x}^{m-1} < \mathbf{0}$ (resp. $\mathcal{A}\mathbf{x}^{m-1} \leq \mathbf{0}$). Choose any real vector $\mathbf{y} \in \mathbb{R}^{n}$ satisfying $\mathbf{y}^{m-1} = P\mathbf{x}^{m-1}$. Then $\mathbf{y > 0}$, and we have $P\mathcal{A}P^{T}(\mathbf{y}^{m-1}) = P\mathcal{A} P^{T}(P\mathbf{x}^{m-1}) = P\mathcal{A}(P^{T}P)\mathbf{x}^{m-1} = P\mathcal{A}\mathbf{x}^{m-1} < \mathbf{0}$ (resp. $P\mathcal{A}P^{T}(\mathbf{y}^{m-1}) \leq \mathbf{0}$).
 
 Moreover, \cite[Theorem 5]{Das} implies that every proper principal subtensor of the real tensor $P\mathcal{A}P^{T}$ is an ($\mathbf{E}$) $\mathbf{E_0}$-tensor. Consequently, the tensor $P\mathcal{A}P^{T}$ is an almost ($\mathbf{E}$) $\mathbf{E_0}$-tensor. The converse part can be deduced from the fact that the tensor $\mathcal{A}$ can be expressed as $\mathcal{A} = P^{T}(P\mathcal{A}P^{T})P$ for a permutation matrix $P$.
\end{proof}

Next, we present the sufficient and the necessary conditions for the underlying tensor to be classified as almost $\mathbf{E_0}$-tensor.

\begin{theorem}
Suppose that every proper principal subtensor of  real tensor $\mathcal{A} \in \mathbb{R}_{m, n}$ is  $\mathbf{E_0}$-tensor. Then, the tensor $\mathcal{A}$ is  almost $\mathbf{E_0}$-tensor iff there is a positive $\mathbf{x} \in \mathbb{R}^n$ such that $\mathbf{x}^{\top} D\left(\mathcal{A} \mathbf{x}^{m-1}\right) < 0$ for any diagonal matrix $O \neq D \in \mathbb{R}^{n \times n}$ with all nonnegative diagonal elements.
\end{theorem}
\begin{proof} 
Let $O \neq D = \text{diag}(d_1, \ldots, d_n)$ be a diagonal matrix with $d_i \geq 0$ for every $i \in I_n$. Note that $d_k > 0$ for some $k \in I_n$. Now, suppose $\mathcal{A} \in \mathbb{R}_{m, n}$ is an almost $\mathbf{E_0}$-tensor. Subsequently, there is a positive $\mathbf{x} \in \mathbb{R}^n$ satisfying $(\mathcal{A} \mathbf{x}^{m-1})_i < 0$ for every $i \in I_n$. Consequently, we obtain 
\[
\mathbf{x}^{\top} D(\mathcal{A}\mathbf{x}^{m-1}) = \sum_{i = 1}^n x_i d_i (\mathcal{A}\mathbf{x}^{m-1})_i < 0.
\]

Conversely, let us suppose there is a positive real vector $\mathbf{x} \in \mathbb{R}^n$ such that $\mathbf{x}^{\top} D\left(\mathcal{A} \mathbf{x}^{m-1}\right) < 0$ for any diagonal matrix $O \neq D \in \mathbb{R}^{n \times n}$ with all of its diagonal elements nonnegative. By choosing different diagonal matrices where only the $i$th diagonal entry $d_i > 0$ and the others are zero, we find that $\mathbf{x}^{\top} D(\mathcal{A}\mathbf{x}^{m-1}) = x_i d_i (\mathcal{A}\mathbf{x}^{m-1})_i < 0$. This condition applies to every $i \in I_n$. Consequently, we observe that $\mathbf{x} > \mathbf{0}$ and $\mathcal{A}\mathbf{x}^{m-1} < \mathbf{0}$. Therefore, $\mathcal{A}$ is an almost $\mathbf{E_0}$-tensor.
\end{proof}

Next, we present both the sufficient and the necessary conditions for a tensor to be classified as an almost ($\mathbf{E}$) $\mathbf{E_0}$-tensor, ensuring that every proper principal subtensor is (strictly) semi positive.

\begin{theorem}
Suppose that every proper principal subtensor of a real tensor $\mathcal{A} \in \mathbb{R}_{m, n}$ is a ($\mathbf{E}$) $\mathbf{E_0}$-tensor. Then the following two statements are considered equivalent:
\begin{itemize}
\item [(i)] $\mathcal{A}$ is an almost ($\mathbf{E}$) $\mathbf{E_0}$-tensor.
\item[(ii)] There is the diagonal tensor $\mathcal{D} \in \mathbb{R}_{m, n}$ having all its diagonal elements positive (nonnegative), ensuring that when added to $\mathcal{A}$, the resulting tensor $\mathcal{A} + \mathcal{D}$ possesses a positive null vector.
\end{itemize}
\end{theorem}
\begin{proof}
This result represents the contrapositive statement of \cite[Theorem 7]{Das}.
\end{proof}

Next, we show the construction of a tensor that is both an almost $\mathbf{E}$-tensor and a completely $\mathbf{S_{0}}$-tensor from an almost $\mathbf{E_0}$-tensor.

\begin{theorem}
If $\mathcal{A} \in \mathbb{R}_{m, n}$ is almost $\mathbf{E_0}$-tensor, then there is a diagonal tensor $\mathcal{D} \in \mathbb{R}_{m, n}$ having all positive diagonal elements, resulting in the tensor $\mathcal{A + D} \in \mathbb{R}_{m, n}$ being almost $\mathbf{E}$-tensor and completely $\mathbf{S_{0}}$-tensor.
\end{theorem}	
\begin{proof}
Let $\mathcal{A}$ be an almost $\mathbf{E_0}$-tensor. Then every proper principal subtensors of tensor $\mathcal{A}$ is a  $\mathbf{E_0}$-tensor, and there is a positive real vector $\mathbf{x} \in \mathbb{R}^n$ such that $(\mathcal{A} \mathbf{x}^{m-1}) < 0$.

Now, consider a diagonal tensor $\mathcal{D}$ with its diagonal elements defined by
\[
d_{ii\ldots i} =  \frac{-(\mathcal{A} \mathbf{x}^{m-1})_i}{x_i\ldots x_i} = \frac{-(\mathcal{A} \mathbf{x}^{m-1})_i}{x_i^{m-1}} \quad \text{for each } i \in I_n.
\]
Importantly, these diagonal components satisfy $d_{ii\ldots i} > 0$. This leads to positive diagonal entries in the tensor $\mathcal{D}$, and consequently, $(\mathcal{D} \mathbf{x}^{m-1})_i = d_{ii\ldots i}x_{i}^{m-1} = -(\mathcal{A} \mathbf{x}^{m-1})_i$, that is, a positive $\mathbf{x} \in \mathbb{R}^n$ satisfying
		\begin{equation}\label{diag}
			(\mathcal{A + D}) \mathbf{x}^{m-1} = \mathbf{0}.
		\end{equation}
The structure of $(\mathcal{A + D})_r^J = \mathcal{A}_r^J + \mathcal{D}_r^J$ showcases that it is a $\mathbf{E}$-tensor for every index set $J$, where $J$ is a proper subset of $I_n$ with $|J| = r$ $(\text{where }1 \leq r < n)$. This is because for every proper subset $J$, $\mathcal{A}_r^J$ is a $\mathbf{E_0}$-tensor and the diagonal tensor $\mathcal{D}_r^J$ is nonnegative, having positive diagonal entries. Then the equation (\ref{diag}) implies that tensor $\mathcal{A + D}$ is an almost $\mathbf{E}$-tensor, as each proper principal subtensors is a $\mathbf{E}$-tensor. Furthermore, from the equation (\ref{diag}), the tensor $\mathcal{A + D}$ is  $\mathbf{S_0}$-tensor. Consequently, $\mathcal{A + D}$ is  completely $\mathbf{S_{0}}$-tensor. This concludes the proof of the theorem.
\end{proof}


\section{Almost (strictly) copositive tensors}\label{sec: Main results}

In this section, we initially present the notion of almost (strictly) copositive tensor, which naturally extends that of an almost (strictly) copositive matrix \cite{H,val}. Tensors of this nature play a pivotal role in formulating criteria for copositivity.

\begin{definition}
A tensor $\mathcal{A} \in \mathbb{R}_{m, n}$ is defined as almost (strictly) copositive if $\mathcal{A}$ is not a (strictly) copositive tensor but each of its proper principal subtensors is (strictly) copositive.
We refer to such tensors as almost ($\mathbf{C}$) $\mathbf{C_0}$-tensors.
\end{definition}

We construct a tensor which is an {\it almost (strictly) copositive} tensor.

\begin{example}
Consider a tensor $\mathcal{A}$ represented by $\left(a_{i_1 i_2 i_3}\right) \in \mathbb{R}_{3,2}$, where specific entries are defined such that $a_{111}=1$,  $a_{212}=-3$, $a_{222}=1$, $a_{112}=-2$, and all the remaining entries are zero. For any vector $\mathbf{x} = (x_1, x_2)^{T} \in \mathbb{R}^2$,
$$
\mathcal{A} \mathbf{x}^2=\left(\begin{array}{l}
x_1^2-2x_1 x_2 \\
x_2^2-3x_1 x_2
\end{array}\right).
$$
We observe that every proper principal subtensor of tensor $\mathcal{A}$ is a ($\mathbf{C}$) $\mathbf{C_0}$-tensor. Additionally, when $\mathbf{x} = (1, 1)^{T} > \mathbf{0}$, we have $\mathcal{A} \mathbf{x}^3 = -3 < \mathbf{0}$. Consequently, $\mathcal{A}$ is an almost ($\mathbf{C}$) $\mathbf{C_0}$-tensor.
\end{example}

Next, assuming symmetry, we show the equivalence between almost copositive and semi-positive tensors. We will depend on the following results for our subsequent analysis.

\begin{theorem}[\cite{Ysong}]\label{c1} 
Each copositive tensor is a semi positive tensor. However, in general, the converse statement does not hold true.
\end{theorem}

\begin{theorem}[\cite{Ysong}]\label{c2} 
	If $\mathcal{A} \in \mathbb{R}_{m, n}$ is symmetric and semi positive tensor, then it is a copositive tensor.
\end{theorem}

The following statement represents that when symmetry is assumed, the classes of almost $\mathbf{E_0}$-tensors and almost $\mathbf{C_0}$-tensors are equivalent.

\begin{theorem}\label{l1}
For a real tensor $\mathcal{A} \in S_{m, n}$, the following two conditions are equivalent:
	\begin{itemize}
	\item[(i)] tensor $\mathcal{A}$ is almost $\mathbf{E_0}$.
	\item[(ii)] tensor$\mathcal{A}$ is almost $\mathbf{C_0}$.
\end{itemize}
\end{theorem}
\begin{proof}
$(i)$ $\Rightarrow$ $(ii)$. 
Let tensor $\mathcal{A}$ be an almost $\mathbf{E_0}$-tensor. Consequently, every proper principal subtensors of tensor $\mathcal{A}$ is a  $\mathbf{E_0}$-tensor, and there is a positive real vector $\mathbf{x} \in \mathbb{R}^n$ such that $(\mathcal{A} \mathbf{x}^{m-1}) < 0$. This implies that for a vector $\mathbf{x > 0}$ we have,
 $$
\mathcal{A} \mathbf{x}^m:=\mathbf{x}^{\top}\left(\mathcal{A} \mathbf{x}^{m-1}\right)=\sum_{i_1, \ldots, i_m=1}^n a_{i_1 \ldots i_m} x_{i_1} \ldots x_{i_m} < 0 .
$$
Further, by Theorem \ref{c2}, every symmetric semi positive tensor is copositive imply that all proper principal subtensors are copositive. Hence tensor $\mathcal{A}$ is an almost $\mathbf{C_0}$-tensor.
	
Now, to show $(ii)$ $\Rightarrow$ $(i)$, assume that tensor $\mathcal{A} \in \mathbb{R}_{m, n}$ is an almost $\mathbf{C_0}$-tensor. Then every proper principal subtensor of tensor $\mathcal{A}$ is a $\mathbf{C_0}$-tensor and from Theorem \ref{c1}, every proper principal subtensor of $\mathcal{A}$ is a $\mathbf{E_0}$-tensor. Now, by \ref{c2}, if symmetric tensor $\mathcal{A}$ is not a $\mathbf{C_0}$-tensor, it implies that $\mathcal{A}$ is not a $\mathbf{E_0}$-tensor. Hence, $\mathcal{A}$ is an almost $\mathbf{E_0}$-tensor.
\end{proof}

\begin{remark}
It is important to observe that the aforementioned result loses validity upon removing the symmetry assumption. The subsequent illustrations involve two instances of non-symmetric tensors, serving to affirm our assertion. 
\end{remark}
\begin{example}
Consider a tensor $\mathcal{A}$ represented by $\left(a_{i_1 i_2 i_3}\right) \in \mathbb{R}_{3,3}$, where specific entries are defined such that $a_{111}=a_{332}=a_{333}=1, a_{112}=a_{331}=-2, a_{223}=-1$, and all the remaining $a_{i_1 i_2 i_3}=0$. Then, tensor $\mathcal{A}$ is an almost $\mathbf{E_0}$-tensor but it is not an almost $\mathbf{C_0}$-tensor. Note that for \(r = 1\), there are three one-dimension principal subtensors of \(\mathcal{A}\). Let \(J_1 = \{1\}\), \(J_2 = \{2\}\), and \(J_3 = \{3\}\). The one-dimension principal subtensors of \(\mathcal{A}\) are \(\mathcal{A}_1^{J_1} = (a_{111}) = (1)\), \(\mathcal{A}_1^{J_2} = (a_{222}) = (0)\), and \(\mathcal{A}_1^{J_3} = (a_{333}) = (1)\). For \(r = 2\), there are three two-dimensional principal subtensors of \(\mathcal{A}\). Let \(J_4 = \{1,2\}\), \(J_5 = \{1,3\}\), and \(J_6 = \{2,3\}\). The two-dimensional principal subtensors of \(\mathcal{A}\) are \(\mathcal{A}_2^{J_4}\), \(\mathcal{A}_2^{J_5}\), and \(\mathcal{A}_2^{J_6}\).
\begin{itemize}
    \item Here \(\mathcal{A}_2^{J_4} = (a_{i_1 i_2 i_3})\) for all \(i_1, i_2, i_3 \in J_4\). Then \(\mathcal{A}_2^{J_4}\) is given by \(a_{111} = 1\), \(a_{112} = -2\), and all the remaining entries of \(\mathcal{A}_2^{J_4}\) are zero. Subsequently, for any vector \(\mathbf{x} = (x_1, x_2)^{T} \in \mathbb{R}^2\),
\[\mathcal{A}_2^{J_4} \mathbf{x}^2 = \begin{pmatrix}x_1^2 - 2x_1 x_2 \\0\end{pmatrix}.\]
    \item Here \(\mathcal{A}_2^{J_5} = (a_{i_1 i_2 i_3})\) for all \(i_1, i_2, i_3 \in J_5\). Then \(\mathcal{A}_2^{J_5}\) is given by \(a_{111} = 1\), \(a_{331} = -2\), \(a_{333} = 1\), and all the remaining entries of \(\mathcal{A}_2^{J_5}\) are zero. Subsequently, for any vector \(\mathbf{x} = (x_1, x_2)^{T} \in \mathbb{R}^2\),
\[\mathcal{A}_2^{J_5} \mathbf{x}^2 = \begin{pmatrix}x_1^2 \\-2x_1 x_3 + x_3^2\end{pmatrix}.\]
    \item Here \(\mathcal{A}_2^{J_6} = (a_{i_1 i_2 i_3})\) for all \(i_1, i_2, i_3 \in J_6\). Then \(\mathcal{A}_2^{J_6}\) is given by \(a_{223} = -1\), \(a_{323} = 1\), \(a_{333} = 1\), and all the remaining entries of \(\mathcal{A}_2^{J_6}\) are zero. Subsequently, for any vector \(\mathbf{x} = (x_1, x_2)^{T} \in \mathbb{R}^2\),
\[\mathcal{A}_2^{J_6} \mathbf{x}^2 = \begin{pmatrix}-x_2 x_3 \\x_2 x_3 + x_3^2\end{pmatrix}.\]
\end{itemize}
Here, each proper principal subtensor of one and two dimensions of tensor \(\mathcal{A}\) is $\mathbf{E_0}$-tensor. For \(r = 3\), tensor \(\mathcal{A}\) is only one principal subtensor of three dimensions. Subsequently, for any vector \(\mathbf{x} = (x_1, x_2)^{T} \in \mathbb{R}^2\),
\[\mathcal{A} \mathbf{x}^2 = \begin{pmatrix}
x_1^2-2x_1 x_2 \\
  -x_1 x_2\\
-2x_1x_3+x_2x_3+x_3^2.
\end{pmatrix}.\]
Let \(\mathbf{y} = (1, 1, 1/2)^{T}\), then \(\mathbf{y} > 0\) and \(\mathcal{A} \mathbf{y}^2 < 0\), and hence, \(\mathcal{A}\) is an almost $\mathbf{E_0}$-tensor. But, tensor \(\mathcal{A}\) is not an almost $\mathbf{C_0}$-tensor as for \(J_4 = \{1, 2\}\), a two-dimensional proper principal subtensor \(\mathcal{A}_2^4\) of tensor \(\mathcal{A}\) is not a $\mathbf{C_0}$-tensor since for \(\mathbf{y} = (1, 1)^T > 0\), \(\mathcal{A}_2^{J_4}\mathbf{y}^3 = -1 < 0\).

\end{example} 
\begin{example}
Consider a tensor $\mathcal{A}$ represented by $\left(a_{i_1 i_2 i_3}\right) \in \mathbb{R}_{3,2}$, where specific entries are defined such that $a_{111}=1, a_{112}=-2,$ and the remaining elements of tensor $\mathcal{A}$ are zero. For \( r = 1 \), it is important to note that there are two one-dimensional principal subtensors of tensor \( \mathcal{A} \). Assume that \(J_1 = \{1\}\) and \(J_2 = \{2\}\). Then, the one-dimensional principal sub-tensors of \(\mathcal{A}\) are \(\mathcal{A}_1^{J_1} = (a_{111}) = (1)\) and \(\mathcal{A}_1^{J_2} = (a_{222}) = (0)\). 

Note that each proper principal subtensor of tensor \(\mathcal{A}\), which is of one dimension, is $\mathbf{C_0}$-tensor, implying that each such subtensor is $\mathbf{E_0}$-tensor. For \(r = 2\), tensor \(\mathcal{A}\) is only one principal subtensor of two dimensions. Subsequently, for any vector \(\mathbf{x} = (x_1, x_2)^{T} \in \mathbb{R}^2\),
\[
\mathcal{A} \mathbf{x}^2 = \begin{pmatrix}
x_1^2 - 2x_1 x_2 \\
0
\end{pmatrix}.
\]
Let $\mathbf{y} = (1, 2)^{T}$, then $\mathbf{y} > 0$ and $\mathcal{A} \mathbf{y}^3 = -3 < 0$. Therefore, $\mathcal{A}$ is not a $\mathbf{C_0}$-tensor, that is, tensor $\mathcal{A}$ is almost $\mathbf{C_0}$. However, tensor $\mathcal{A}$ is not almost $\mathbf{E_0}$-tensor. Since tensor $\mathcal{A}$ is $\mathbf{E_0}$-tensor as it has a nonnegative row subtensor.
\end{example}	

Analogously to the proof presented in Theorem \ref{l1}, the subsequent theorem can be established.
\begin{theorem}\label{mr1}
Suppose $\mathcal{A}$ is a real tensor in $S_{m, n}$. Then, tensor $\mathcal{A}$ is almost $\mathbf{E}$-tensor iff it is an almost $\mathbf{C}$-tensor.
\end{theorem}

 Now, we shift our focus to the eigenvalues of these special structured tensors. In the next result, we establish a relation between almost ($\mathbf{C}$) $\mathbf{C_0}$-tensors and $H^{++}$-eigenvalues of tensors.

\begin{lemma}\label{mr2}
For a real tensor $\mathcal{A}$ in $S_{m, n}$ that is the ($\mathbf{C}$) $\mathbf{C_0}$-tensor of order $m$ and dimension $n - 1$, the following two statements are equivalent:
		\begin{itemize}
			\item [(i)] $\mathcal{A}$ is not a ($\mathbf{C}$) $\mathbf{C_0}$-tensor.
			\item[(ii)] There exists a (nonpositive) negative $H^{++}$-eigenvalue.
		\end{itemize}
	\end{lemma}
 \begin{proof}
    $(i) \iff(ii)$. By \cite[Theorems 4.1 and 4.2]{nsc}, $\mathcal{A}$ is not a ($\mathbf{C}$) $\mathbf{C_0}$-tensor if and only if a principal subtensor of tensor $\mathcal{A}$ has a (nonpositive) negative $H^{++}$-eigenvalue, that is, a principal subtensor of tensor $\mathcal{A}$ has a (nonpositive) negative eigenvalue along with the corresponding eigenvector, which is positive. As tensor $\mathcal{A}$ is a ($\mathbf{C}$) $\mathbf{C_0}$-tensor of order $m$ and dimension $n-1$, such a principal subtensor has to be $\mathcal{A}$ itself.
 \end{proof}

In the next result, we establish a relation between the almost ($\mathbf{E}$) $\mathbf{E_0}$-tensors and $H^{++}$-eigenvalues of tensors under symmetric assumption.
\begin{theorem}
For a symmetric tensor $\mathcal{A} \in \mathbb{R}_{m, n}$, if it is an almost ($\mathbf{E}$) $\mathbf{E_0}$-tensor, then it possesses a negative (nonpositive) $H^{++}$-eigenvalue.
\end{theorem}
\begin{proof}
Assume tensor $\mathcal{A}$ is a symmetric almost ($\mathbf{E}$) $\mathbf{E_0}$-tensor. By Theorems \ref{l1} and \ref{mr1}, tensor $\mathcal{A}$ is an almost ($\mathbf{E}$) $\mathbf{E_0}$-tensor. Consequently, the result is derived from Lemma \ref{mr2}.
\end{proof}



\section{Conclusions}
This paper offers a theoretical framework for detecting and constructing (strictly) semi-positive tensors, particularly focusing on cases where every proper principal subtensor belongs to these specialized tensor classes.

\section*{Acknowledgement}


\end{document}